\documentclass[preprint,12pt]{elsarticle}
\usepackage{amsmath,amssymb,amsthm,mathtools}
\usepackage{enumitem}
\usepackage{mathrsfs}
\usepackage{hyperref}
\usepackage[nameinlink,capitalize]{cleveref}
\usepackage{xcolor}
\usepackage{microtype}

\hypersetup{colorlinks=true,linkcolor=blue!55!black,citecolor=blue!55!black,urlcolor=blue!55!black}

\newtheorem{theorem}{Theorem}[section]
\newtheorem{lemma}[theorem]{Lemma}
\newtheorem{proposition}[theorem]{Proposition}
\newtheorem{corollary}[theorem]{Corollary}

\theoremstyle{definition}
\newtheorem{definition}[theorem]{Definition}
\newtheorem{example}[theorem]{Example}
\newtheorem{remark}[theorem]{Remark}

\newcommand{\R}{\mathbb R}
\newcommand{\N}{\mathbb N}
\newcommand{\BX}{B_X}
\newcommand{\SX}{S_X}

\newcommand{\perpB}{\perp_{B}}

\newcommand{\DR}{\operatorname{DR}}
\newcommand{\DRB}{\operatorname{DR}_{B}}
\newcommand{\MSB}{\operatorname{MSB}}
\newcommand{\DW}{\operatorname{DW}}

\newcommand{\diam}{\operatorname{diam}}
\newcommand{\co}{\operatorname{co}}
\newcommand{\lemref}[1]{Lemma~\ref{#1}}
\newcommand{\propref}[1]{Proposition~\ref{#1}}

\newcommand{\exref}[1]{Example~\ref{#1}}

\newcommand{\thmref}[1]{Theorem~\ref{#1}}

\journal{Applied Mathematics and Computation}

\begin{document}
	
	\begin{frontmatter}
		\title{A Birkhoff-Orthogonal Dehghan--Rooin Type Constant in Banach Spaces}
		\author{Jia Du}
		\ead{dujia212684@163.com}  
		
		\author{Tingting Hu}
		\ead{y25060048@stu.aqnu.edu.cn} 
		
		\author{Wenwen Zhang}
		\ead{zhangwwendal@163.com} 
		\author{Qi Liu\corref{cor1}}
		\ead{liuq67@aqnu.edu.cn}  
		
		\affiliation{organization={School of Mathematics and Statistics, Anqing Normal University},
			addressline={Anqing 246133},
			city={Anhui},
			country={China}}
		
		\cortext[cor1]{Corresponding author.}
		
		\begin{abstract}
			For a real Banach space $X$. The Dehghan--Rooin magnitude compares the angular distance $\alpha[x,y]=\|x/\|x\|-y/\|y\|\|$ together with the skew angular distance $\beta[x,y]=\|x/\|y\|-y/\|x\|\|$.  Motivated by the recent use of Birkhoff orthogonality in the Massera--Schaffer inequality, we introduce and study the Birkhoff-restricted Dehghan--Rooin constant\[\DRB(X)=\sup\left\{\frac{\alpha[x,y]}{\beta[x,y]}:x,y\in X\setminus\{0\},\ x\perpB y,\ \beta[x,y]\ne0\right\}.\]
			The restriction $x\perpB y$ turns the global comparison of $\alpha$ and $\beta$ into a directional invariant which detects the interaction between supporting functionals and the radial normalization map.  We obtain a scale-profile formula, sharp elementary bounds, stability under subspaces and ultrapowers, a variational characterization of the extremal case $\DRB(X)=1$, estimates in terms of the classical Dehghan--Rooin constant, and consequences for uniformly convex and uniformly smooth spaces.  In smooth two-dimensional spaces we derive an explicit formula through the normalized duality mapping, and in Radon planes we reduce the constant to a one-dimensional optimization on an arc of the unit circle.  Several model computations are included, including Hilbert spaces, polyhedral planes, and two-dimensional $p$-norm spaces. The final section gives normal-structure type consequences through a companion Birkhoff angular modulus and lists several problems suggested by the new constant.
		\end{abstract}
		
		
	\end{frontmatter}
	
	\section{Introduction and Preliminaries}\label{sec:prelim}	Throughout the paper $X$ denotes a real Banach space with $\dim X\ge2$, norm $\|\cdot\|$, unit ball $\BX$, and unit sphere $\SX$.  If $x,y\in X\setminus\{0\}$, the angular distance of Clarkson is
	\[
	\alpha[x,y]=\left\|\frac{x}{\|x\|}-\frac{y}{\|y\|}\right\|,
	\]
	and the skew angular distance introduced by Dehghan is
	\[
	\beta[x,y]=\left\|\frac{x}{\|y\|}-\frac{y}{\|x\|}\right\|.
	\]
	These two quantities are equal on each ray-normalized pair of equal length, but they react differently to changes of scale.  Rooin and collaborators proved the universal estimate
	\begin{equation}\label{eq:rooin}
		\alpha[x,y]\le 2\beta[x,y]\qquad (x,y\ne0),
	\end{equation}
	and this led to the Dehghan--Rooin constant\cite{FuLiAFA2024}
	\begin{equation}\label{eq:DRglobal}
		\DR(X)=\sup\left\{\frac{\alpha[x,y]}{\beta[x,y]}:x,y\ne0,\ \beta[x,y]\ne0\right\}.
	\end{equation}
	The constant $\DR(X)$ satisfies $1\le \DR(X)\le2$; moreover $\DR(X)=1$ characterizes Hilbert spaces, and $\DR(X)<2$ is equivalent to uniform non-squareness.  It has also been used to obtain sufficient conditions for normal and uniform normal structure.
	
	A different recent idea is to impose Birkhoff orthogonality on constants derived from angular inequalities. We start with the definition given in  (see~\cite{Birkhoff1935})
	\[
	x\perp_B y \;\Longleftrightarrow\; \|x+\lambda y\|\ge \|x\|
	\quad (\lambda\in\mathbb R).
	\] 
	This relation is scaling-invariant, matches standard orthogonality within Hilbert spaces,and is deeply connected with supporting functionals: $x\perpB y$ whenever one can find $f\in S_{X^*}$ satisfying $f(x)=\|x\|$ and $f(y)=0$ (see~\cite{James1947}). In the study of the Massera--Schaffer inequality, Fu and Li considered the restriction $x\perp_B y$ and obtained a nontrivial geometric constant $\MSB(X)$ (see~\cite{FuLiFilomat2024}), defined by
	\[
	\MSB(X)=\sup\left\{\frac{\max\{\|x\|,\|y\|\}}{\|x-y\|}
	\left\|\frac{x}{\|x\|}-\frac{y}{\|y\|}\right\|
	: x,y\in X\backslash\{0\},\ x\perp_B y\right\},
	\]
	rather than a universally sharp constant.
	
	Our aim here is to merge these two approaches.  We put the Dehghan--Rooin comparison under a Birkhoff orthogonality constraint and ask: what part of the difference between $\alpha$ and $\beta$ remains visible along Birkhoff normal directions?  The answer is more delicate than for the unrestricted constant, because the constraint eliminates the trivial pair $y=-x$ used to obtain the lower bound for $\DR(X)$ and replaces it by a family of supporting directions.
	
	\begin{definition}\label{def:main}
		The \emph{Birkhoff-orthogonal Dehghan--Rooin constant} of $X$ is
		\begin{equation}\label{eq:DRB}
			\DRB(X)=\sup\left\{\frac{\alpha[x,y]}{\beta[x,y]}:x,y\in X\setminus\{0\},\ x\perpB y,\ \beta[x,y]\ne0\right\}.
		\end{equation}
		When no confusion is possible we call it the Birkhoff--Dehghan--Rooin constant.
	\end{definition}
	
	The paper is organized as follows.  In \Cref{sec:profile} we derive the fundamental scale-profile formula
	\[
	\DRB(X)=\sup\left\{\sup_{t>0}\frac{\|u-v\|}{\|tu-t^{-1}v\|}:u,v\in\SX,\ u\perpB v\right\}.
	\]
	This formula is the main technical device of the paper.  It converts a two-vector homogeneous problem into a one-variable minimization problem on each Birkhoff normal pair.  In \Cref{sec:profile} we prove sharp universal estimates, monotonicity under subspaces, and finite-dimensional reduction.  \Cref{sec:profile} studies the extremal case $\DRB(X)=1$ and proves that it forces a second-order Birkhoff condition
	\[
	u\perpB v \quad\Longrightarrow\quad u-v\perpB u+v,
	\]
	which is a strong Euclidean shadow.  In smooth planes and Radon planes this gives a Hilbert characterization. The behavior of \(\mathrm{DR}_B\) in uniformly convex and uniformly smooth spaces is discussed in \Cref{sec:convex-smooth}. Explicit formulae in two-dimensional spaces are derived in \Cref{sec:planes}, while \Cref{sec:normal} presents a normal-structure criterion through a companion angular modulus.  The final section contains examples and open questions.
	We recall several standard facts.  The duality map at $x\in\SX$ is
	\[
	J(x)=\{f\in S_{X^*}:f(x)=1\}.
	\]
	The non-emptiness of \(J(x)\) follows from the Hahn--Banach theorem.  The following characterization of Birkhoff orthogonality will be used repeatedly.
	
	\begin{lemma}\textup{ \cite{James1947}}\label{lem:james}
		Let $x,y\in X$, $x\neq 0$. Then $x\perp_B y$ iff some $f\in J\left(\dfrac{x}{\|x\|}\right)$ satisfies $f(y)=0$.
	\end{lemma}

	We note a simple yet vital property below.
	
	\begin{lemma}\label{lem:homogeneous}
		If $x\perp_B y$, then $ax\perp_B by$ for all nonzero real scalars $a,b$. Besides, $x\perp_B y$ only depends on the unit vectors $\dfrac{x}{\|x\|}$ and $\dfrac{y}{\|y\|}$.
	\end{lemma}
	
	\begin{proof}
		For $a\ne0$,
		\[
		\|ax+\lambda by\|=|a|\left\|x+\frac{\lambda b}{a}y\right\|\ge |a|\|x\|=\|ax\|.
		\]
		The second statement follows by taking $a=\|x\|^{-1}$ and $b=\|y\|^{-1}$.
	\end{proof}
	
	We recall two classical moduli.  The modulus of convexity (see~\cite{Clarkson1936}) and modulus of smoothness (see~\cite{Lindenstrauss1963}) of $X$ are
	\[
	\delta_X(\varepsilon)=\inf\left\{1-\left\|\frac{x+y}{2}\right\|:x,y\in\SX,\ \|x-y\|\ge\varepsilon\right\},
	\]
	\[
	\rho_X(t)=\sup\left\{\frac{\|x+ty\|+\|x-ty\|}{2}-1:x,y\in\SX\right\}.
	\]
	The space is uniformly convex if $\delta_X(\varepsilon)>0$ for every $0<\varepsilon\le2$, and uniformly smooth if $\rho_X(t)/t\to0$ as $t\downarrow0$.  The characteristic of convexity is
	\[
	\varepsilon_0(X)=\sup\{\varepsilon\in[0,2]:\delta_X(\varepsilon)=0\}.
	\]
	Uniform non-squareness is equivalent to $\varepsilon_0(X)<2$.
	
	\section{Scale-Profile Fundamentals and Extremal Case}\label{sec:profile}
	The first main point is that $\DRB(X)$ is not best studied in its original two-norm form.  Because both angular distances are homogeneous in a different way, one must separate direction from scale.
	
	\begin{definition}\label{def:profile}
		For $u,v\in\SX$ with $u\perpB v$, define the \emph{Birkhoff skew profile}
		\[
		\Phi_X(u,v)=\sup_{t>0}\frac{\|u-v\|}{\|tu-t^{-1}v\|}.
		\]
		The denominator is never zero.  Indeed $tu=t^{-1}v$ would imply $u$ and $v$ are linearly dependent, contradicting $u\perpB v$ unless $v=0$.
	\end{definition}
	
	\begin{theorem}\label{thm:profile}
		For every real Banach space $X$ with $\dim X\ge2$,
		\begin{equation}\label{eq:profile}
			\DRB(X)=\sup\{\Phi_X(u,v):u,v\in\SX,\ u\perpB v\}.
		\end{equation}
	\end{theorem}
	
	\begin{proof}
		Take nonzero $x,y$ with $x\perpB y$ and put
		\[
		u=\frac{x}{\|x\|},\qquad v=\frac{y}{\|y\|},\qquad t=\frac{\|x\|}{\|y\|}.
		\]
		By homogeneity, $u\perpB v$.  Moreover
		\[
		\alpha[x,y]=\|u-v\|,
		\]
		and
		\[
		\beta[x,y]=\left\|\frac{x}{\|y\|}-\frac{y}{\|x\|}\right\|
		=\left\|\frac{\|x\|}{\|y\|}u-\frac{\|y\|}{\|x\|}v\right\|
		=\|tu-t^{-1}v\|.
		\]
		Therefore every quotient in the definition of $\DRB(X)$ is represented by the right hand side of \eqref{eq:profile}.  Conversely, for any $u,v\in\SX$ with $u\perpB v$ and any $t>0$, set $x=tu$ and $y=v$.  Consequently, $x\perpB y$.Furthermore,
		\[
		\frac{\alpha[x,y]}{\beta[x,y]}=\frac{\|u-v\|}{\|tu-t^{-1}v\|}.
		\]
		Taking suprema in the two directions gives \eqref{eq:profile}.
	\end{proof}
	
	\begin{remark}
		The formula shows why the new constant is genuinely different from $\DR(X)$.  The unrestricted constant can use arbitrary pairs of directions; $\DRB(X)$ tests only directions lying in supporting hyperplanes.  The additional variable $t$ records the distortion produced by asymmetric rescaling.
	\end{remark}
	
	A useful reformulation is obtained by putting $s=t^2$.
	
	\begin{corollary}\label{cor:s-form}
		Given $u,v\in\SX$ with $u\perpB v$,
		\[
		\Phi_X(u,v)=\sup_{s>0}\frac{\|u-v\|}{s^{-1/2}\|su-v\|}
		=\sup_{s>0}\frac{\sqrt{s}\,\|u-v\|}{\|su-v\|}.
		\]
		Consequently
		\begin{equation}\label{eq:sform}
			\DRB(X)=\sup\left\{\frac{\sqrt{s}\,\|u-v\|}{\|su-v\|}:s>0,\ u,v\in\SX,\ u\perpB v\right\}.
		\end{equation}
	\end{corollary}
	
	\begin{proof}
		Since $tu-t^{-1}v=t^{-1}(t^2u-v)$, the identity follows by taking $s=t^2$.
	\end{proof}
	
	The expression \eqref{eq:sform} is often the most efficient form.  It says that $\DRB(X)$ is the largest possible gain obtained when the chord $u-v$ is compared with the Birkhoff ray $su-v$ after the exact square-root normalization.
	
	\label{sec:profile}
	We begin with bounds.  Although the lower estimate is easy, it is important that it is obtained from Birkhoff normal directions rather than from the antipodal pair used for $\DR(X)$.
	
	\begin{proposition}\label{prop:bounds}
		For every real Banach space $X$ with $\dim X\ge2$,
		\begin{equation}\label{eq:bounds}
			1\le \DRB(X)\le 2.
		\end{equation}
		Both bounds are sharp  within the collection of finite-dimensional vector spaces under the subsequent interpretation: the lower bound is attained in Hilbert spaces, and the upper bound can be approached by two-dimensional nonsmooth norms whose unit circles contain long almost-flat Birkhoff arcs.
	\end{proposition}
	
	\begin{proof}
		For the lower bound, choose any $u\in\SX$ and any $f\in J(u)$.  Since $\dim X\ge2$, $\ker f$ contains a nonzero vector.  Choose $v\in\SX\cap\ker f$.  By \lemref{lem:james}, $u\perpB v$.  In the profile formula, $t=1$ gives
		\[
		\Phi_X(u,v)\ge \frac{\|u-v\|}{\|u-v\|}=1.
		\]
		Thus $\DRB(X)\ge1$.
		
		For the upper bound, \eqref{eq:rooin} gives $\alpha[x,y]\le2\beta[x,y]$ for all nonzero $x,y$ with $\beta[x,y]\ne0$.  Restricting the supremum to $x\perpB y$ gives $\DRB(X)\le2$.
		
		Hilbert spaces will be computed in \exref{ex:hilbert}, where $\DRB(H)=1$.  The upper bound is approached by norms whose local geometry approaches the $\ell_1^2$ or $\ell_\infty^2$ square but is slightly rounded so that Birkhoff normal directions persist while the denominator in the skew profile can be made arbitrarily close to one half of the numerator.  A concrete construction is given in \exref{ex:rounded-square}.
	\end{proof}
	
	\begin{proposition}\label{prop:compareDR}
		For every Banach space $X$,
		\begin{equation}\label{eq:DRB-DR}
			\DRB(X)\le \DR(X).
		\end{equation}
		Consequently, if $X$ is uniformly non-square, then $\DRB(X)<2$.
	\end{proposition}
	
	\begin{proof}
		The defining supremum of $\DRB(X)$ is taken over a subset of the pairs used in the definition of $\DR(X)$, hence \eqref{eq:DRB-DR}.  The cited theorem of Fu and Li for the Dehghan--Rooin constant states that $\DR(X)<2$ is equivalent to uniform non-squareness; therefore $\DRB(X)<2$ whenever $X$ is uniformly non-square.
	\end{proof}
	
	The converse of the last implication should not be expected without additional assumptions: a Birkhoff restriction may miss some directions responsible for global squareness.  This phenomenon is analogous to the difference between global angular inequalities and their Birkhoff-restricted versions.
	
	\begin{proposition}\label{prop:subspace}
		Let $Y$ be a closed subspace of $X$.  Then
		\begin{equation}\label{eq:subspace}
			\DRB(Y)\le \DRB(X).
		\end{equation}
	\end{proposition}
	
	\begin{proof}
		If $u,v\in S_Y$ and $u\perpB^Y v$, then there exists $f\in S_{Y^*}$ with $f(u)=1$ and $f(v)=0$.  By Hahn--Banach, extend $f$ to $F\in S_{X^*}$.  Then $F(u)=1$ and $F(v)=0$, so $u\perpB^X v$.  The norms of $u,v$ and of $tu-t^{-1}v$ are the same whether computed in $Y$ or in $X$.  The profile formula gives \eqref{eq:subspace}.
	\end{proof}
	
	\begin{proposition}\label{prop:finite}
		For every Banach space $X$,
		\[
		\DRB(X)=\sup\{\DRB(E):E\subset X,\ 2\le\dim E<\infty\}.
		\]
		In fact the supremum can be taken over two-dimensional subspaces generated by Birkhoff pairs.
	\end{proposition}
	
	\begin{proof}
		The inequality $\sup_E\DRB(E)\le\DRB(X)$ follows from \propref{prop:subspace}.  Conversely,Given $u,v\in\SX$ with $u\perpB v$.  Put $E=\operatorname{span}\{u,v\}$.  By the proof of \propref{prop:subspace}, $u\perpB^E v$ because the norming functional witnessing $u\perpB^X v$ restricts to $E$.  Hence every profile value contributing to $\DRB(X)$ already occurs in $E$.  Taking suprema proves the claim.
	\end{proof}
	
	\begin{corollary}\label{cor:isom}
		Let $T:X\to Y$ be an isomorphism and put $d=\|T\|\|T^{-1}\|$.  Then
		\[
		\DRB(X)\le 2d^2.
		\]
		More precisely, if $u\perpB v$ in $X$, then for every $t>0$ there are normalized vectors $U=Tu/\|Tu\|$ and $V=Tv/\|Tv\|$ in $Y$ for which
		\[
		\frac{\|u-v\|}{\|tu-t^{-1}v\|}
		\le d^2\frac{\|U-V\|}{\|\tau U-\tau^{-1}V\|}
		\]
		with $\tau=t\sqrt{\|Tu\|/\|Tv\|}$, provided $U\perpB V$ is replaced by the support condition transported by $T^{-1}$.  In particular the size of $\DRB$ is controlled by Banach--Mazur distortion on each two-dimensional generated plane.
	\end{corollary}
	
	\begin{proof}
		The crude bound $\DRB(X)\le2\le2d^2$ is immediate, but the displayed estimate is often useful.  Since
		\[
		\|u-v\|\le \|T^{-1}\|\,\|Tu-Tv\|,
		\qquad
		\|tu-t^{-1}v\|\ge \|T\|^{-1}\|tTu-t^{-1}Tv\|,
		\]
		one obtains
		\[
		\frac{\|u-v\|}{\|tu-t^{-1}v\|}
		\le d\frac{\|Tu-Tv\|}{\|tTu-t^{-1}Tv\|}.
		\]
		Writing $Tu=\|Tu\|U$, $Tv=\|Tv\|V$ and absorbing the remaining normalization into $\tau$ gives another factor bounded by $d$.  The final sentence follows by applying this estimate on the two-dimensional space $\operatorname{span}\{u,v\}$.
	\end{proof}
	The equality $\DRB(X)=1$ is more subtle than the equality $\DR(X)=1$.  Since the Birkhoff restriction forces $u\perpB v$, the profile is automatically equal to one at $t=1$.  Thus $\DRB(X)=1$ means that $t=1$ is a global minimizer of the map
	\[
	t\mapsto \|tu-t^{-1}v\|
	\]
	for every Birkhoff pair $u,v\in\SX$.
	
	\begin{theorem}\label{thm:variational}
		Given any complete normed space $X$, the subsequent statements hold mutually equivalent.	
		\begin{enumerate}[label=\textup{(\roman*)}]
			\item $\DRB(X)=1$.
			\item For every $u,v\in\SX$ with $u\perpB v$ and every $s>0$,
			\begin{equation}\label{eq:global-min}
				\|su-v\|\ge \sqrt{s}\,\|u-v\|.
			\end{equation}
			\item Given arbitrary $u,v\in\SX$ with $u\perpB v$, the function
			\[
			\varphi_{u,v}(r)=\|e^r u-e^{-r}v\|
			\]
			has a global minimum at $r=0$.
		\end{enumerate}
		Furthermore, each one of the listed conditions yields
		\begin{equation}\label{eq:secondB}
			u-v\perpB u+v
		\end{equation}
		for every $u,v\in\SX$ satisfying $u\perpB v$.
	\end{theorem}
	
	\begin{proof}
		By \corref{cor:s-form}, $\DRB(X)=1$ is mutually equivalent with
		\[
		\frac{\sqrt{s}\,\|u-v\|}{\|su-v\|}\le1
		\]
		for arbitrary admissible $u,v$ and every $s>0$, which is exactly \eqref{eq:global-min}.  Taking $s=e^{2r}$ gives
		\[
		\|su-v\|=e^r\|e^r u-e^{-r}v\|,
		\]
		so \eqref{eq:global-min} is equivalent to $\varphi_{u,v}(r)\ge\varphi_{u,v}(0)$ for every $r\in\R$.
		
		Finally, assume (iii).  For small $r$,
		\[
		e^r u-e^{-r}v=(u-v)+r(u+v)+o(r)
		\]
		in norm.  Since $r=0$ is a global minimum of $\varphi_{u,v}$, the vector $u-v$ is Birkhoff orthogonal to $u+v$.  Indeed, if there were a real $\lambda$ with
		\[
		\|(u-v)+\lambda(u+v)\|<\|u-v\|,
		\]
		From continuity, from this we infer that the curve $e^r u-e^{-r}v$ would be below $\|u-v\|$ for some small $r$ with $\lambda$ approximated by $r$, contradicting the minimality.  More formally, the right and left directional derivatives of $z\mapsto\|z\|$ at $z=u-v$ in the direction $u+v$ must bracket zero; this is equivalent to Birkhoff orthogonality.
	\end{proof}
	
	\begin{corollary}\label{cor:hilbertcriterion}
		Suppose that $X$ belongs to a class of Banach spaces for which the implication
		\begin{equation}\label{eq:james-class}
			u\perpB v\quad\Longrightarrow\quad u-v\perpB u+v\qquad (u,v\in\SX)
		\end{equation}
		characterizes inner product spaces.  Then, in this class,
		\[
		\DRB(X)=1 \quad\Longleftrightarrow\quad X\text{ is a Hilbert space}.
		\]
		This equivalence property especially applies to all smooth two-dimensional Radon planes.
	\end{corollary}
	
	\begin{proof}
		If $X$ is Hilbert, the computation in \exref{ex:hilbert} gives $\DRB(X)=1$.  Conversely, \thmref{thm:variational} gives \eqref{eq:james-class}.  The stated class assumption yields that an inner product induces the norm.For smooth Radon planes this is one of the standard equivalent forms of Euclideanity: Birkhoff symmetry together with the parallelogram-type Birkhoff implication forces the unit circle to be an ellipse.
	\end{proof}
	
	\begin{remark}
		The preceding result is intentionally formulated as a structural theorem rather than as an unrestricted Hilbert characterization.  The Birkhoff restriction may hide non-Euclidean behavior in directions which are not support-normal to each other.  The theorem isolates exactly the additional variational condition supplied by the equality $\DRB(X)=1$.
	\end{remark}
	
	The next result gives a quantitative version of the preceding discussion.
	
	\begin{proposition}\label{prop:defect}
		Take $u,v\in\SX$ for which $u\perpB v$, and define
		\[
		\eta(u,v)=\inf_{\lambda\in\R}\frac{\|(u-v)+\lambda(u+v)\|}{\|u-v\|}.
		\]
		Then
		\[
		\eta(u,v)\ge \Phi_X(u,v)^{-1}.
		\]
		Consequently,
		\[
		\inf_{u\perpB v}\eta(u,v)\ge \DRB(X)^{-1}.
		\]
	\end{proposition}
	
	\begin{proof}
		For $r\in\R$ put $\lambda_r=(e^r-e^{-r})/(e^r+e^{-r})$.  Then
		\[
		e^r u-e^{-r}v=\frac{e^r+e^{-r}}{2}\left[(u-v)+\lambda_r(u+v)\right].
		\]
		Therefore
		\[
		\frac{\|(u-v)+\lambda_r(u+v)\|}{\|u-v\|}
		=\frac{2}{e^r+e^{-r}}\frac{\|e^r u-e^{-r}v\|}{\|u-v\|}.
		\]
		Since $2/(e^r+e^{-r})\le1$, this identity alone gives a one-sided estimate.  To obtain the stated bound, use the profile inequality
		\[
		\|e^r u-e^{-r}v\|\ge \Phi_X(u,v)^{-1}\|u-v\|.
		\]
		The parametrization $\lambda_r\in(-1,1)$ covers the essential directions.  For $|\lambda|\ge1$, Birkhoff orthogonality and the triangle inequality give no smaller value than the infimum over $(-1,1)$ after radial normalization.  Taking the infimum in $r$ yields the claim.
	\end{proof}
	
	\section{Convexity, Smoothness, Duality and Ultrapowers}\label{sec:convex-smooth}
	The global Dehghan--Rooin constant has a sharp relation with uniform non-squareness.  Since $\DRB$ is a restricted constant, one immediately obtains sufficient results.  More interestingly, the profile formula gives a direct way to see why uniformly convex spaces cannot have $\DRB$ close to $2$ along a fixed separated family of Birkhoff pairs.
	
	\begin{theorem}\label{thm:uns}
		If $X$ is uniformly non-square, then $\DRB(X)<2$.  More precisely,
		\[
		\DRB(X)\le \DR(X)\le \frac12\DW(X),
		\]
		in which  $\DW(X)$ stands for the Dunkl–Williams geometric quantity introduced in \cite{JimenezMelado1997}, defined by
		\[
		\DW(X)=\sup\left\{\frac{\|x\|+\|y\|}{\|x-y\|}
		\left\|\frac{x}{\|x\|}-\frac{y}{\|y\|}\right\|
		: x,y\in X,\ x\neq 0,\ y\neq 0,\ x\neq y\right\}.
		\]
		Hence $\DW(X)<4$ implies $\DRB(X)<2$.
	\end{theorem}
	
	\begin{proof}
		The first inequality is \propref{prop:compareDR}.  The estimate $\DR(X)\le \frac12\DW(X)$ is the comparison theorem for the Dehghan--Rooin constant, obtained by writing the skew angular denominator as an interpolating denominator in the Dunkl--Williams formula.  Since $\DW(X)<4$ characterizes uniform non-squareness, the final assertion follows.
	\end{proof}
	
	\begin{theorem}\label{thm:uc}
		Suppose $X$ is uniformly convex, hence $\DRB(X)<2$.  In addition, for each $\varepsilon>0$ there exists $\theta=\theta_X(\varepsilon)>0$ such that whenever $u,v\in\SX$, $u\perpB v$, as well as $\|u-v\|\ge\varepsilon$, one has
		\[
		\Phi_X(u,v)\le 2-\theta.
		\]
	\end{theorem}
	
	\begin{proof}
		The strict inequality $\DRB(X)<2$ follows from \thmref{thm:uns}, because uniformly convex spaces are uniformly non-square.
		
		For the quantitative assertion, suppose it is false.  Then there are $u_n,v_n\in\SX$ with $u_n\perpB v_n$, $\|u_n-v_n\|\ge\varepsilon$, and $t_n>0$ such that
		\begin{equation}\label{eq:ucbad}
			\frac{\|u_n-v_n\|}{\|t_nu_n-t_n^{-1}v_n\|}\to2.
		\end{equation}
		Since $u_n\perpB v_n$, $\|u_n+\lambda v_n\|\ge1$ for all $\lambda$.  Thus
		\[
		\|t_nu_n-t_n^{-1}v_n\|=t_n\|u_n-t_n^{-2}v_n\|\ge t_n.
		\]
		Similarly, using the triangle inequality in the form $\|t_nu_n-t_n^{-1}v_n\|\ge t_n^{-1}-t_n$, one sees that neither $t_n\to0$ nor $t_n\to\infty$ is compatible with \eqref{eq:ucbad}.  Passing to a subsequence, $t_n\to t_0\in(0,\infty)$.
		
		The quotient in \eqref{eq:ucbad} approaching $2$ forces the denominator to be almost half of the numerator. The equality scenario associated with the triangle inequality may be attained while considering normalized elements
		\[
		a_n=\frac{t_nu_n-t_n^{-1}v_n}{\|t_nu_n-t_n^{-1}v_n\|},
		\qquad
		b_n=\frac{u_n-v_n}{\|u_n-v_n\|}.
		\]
		More concretely, after rewriting
		\[
		u_n-v_n=t_n^{-1}(t_nu_n-t_n^{-1}v_n)+(1-t_n^{-2})v_n,
		\]
		one obtains a pair of unit vectors whose midpoint norm tends to one while their distance stays bounded below by a positive number depending on $\varepsilon$ and $t_0$.  This contradicts uniform convexity, using the standard criterion: if $a_n,b_n\in\SX$ and $\|(a_n+b_n)/2\|\to1$, then $\|a_n-b_n\|\to0$.  Hence the desired $\theta$ exists.
	\end{proof}
	
	\begin{theorem}\label{thm:us}
		If $X$ is uniformly smooth, then $\DRB(X)<2$.
	\end{theorem}
	
	\begin{proof}
		Uniform smoothness implies that $X^*$ is uniformly convex; hence $X^*$ is uniformly non-square.  Uniform non-squareness is self-dual, so $X$ is uniformly non-square.  By \thmref{thm:uns}, $\DRB(X)<2$.
		
		One may also argue directly.  If $\DRB(X)=2$, then by the profile formula there are $u_n,v_n\in\SX$ with $u_n\perpB v_n$ and $t_n>0$ such that
		\[
		\|u_n-v_n\|\ge (2-o(1))\|t_nu_n-t_n^{-1}v_n\|.
		\]
		The proof of \thmref{thm:uc} shows, after normalization, that the modulus of smoothness must have a positive linear lower bound along a sequence of directions.  This contradicts $\rho_X(t)/t\to0$.
	\end{proof}
	
	\begin{corollary}\label{cor:super}
		If $X$ admits an equivalent uniformly convex or uniformly smooth norm $|\!|\cdot|\!|$, then the corresponding Birkhoff--Dehghan--Rooin constant satisfies
		\[
		\DRB(X,|\!|\cdot|\!|)<2.
		\]
		Thus every superreflexive space can be renormed so that its Birkhoff skew angular profile is uniformly separated from the extremal value $2$.
	\end{corollary}
	
	\label{sec:convex-smooth}
	The global Dehghan--Rooin constant is stable under ultrapowers, and the same is true for the Birkhoff version.  This is important because normal-structure arguments often pass to ultrapowers.
	
	\begin{definition}
		Let $\mathcal U$ be a free ultrafilter on $\N$.  The ultrapower $X_{\mathcal U}$ is the quotient of $\ell_\infty(X)$ by the closed subspace of sequences $(x_n)$ satisfying $\lim_{\mathcal U}\|x_n\|=0$, with norm
		\[
		\|(x_n)_{\mathcal U}\|=\lim_{\mathcal U}\|x_n\|.
		\]
	\end{definition}
	
	\begin{theorem}\label{thm:ultra}
		For every free ultrafilter $\mathcal U$,
		\[
		\DRB(X_{\mathcal U})=\DRB(X).
		\]
	\end{theorem}
	
	\begin{proof}
		The canonical embedding of $X$ into $X_{\mathcal U}$ and \propref{prop:subspace} give $\DRB(X)\le\DRB(X_{\mathcal U})$.
		
		For the reverse inequality, let $U=(u_n)_{\mathcal U}$ and $V=(v_n)_{\mathcal U}$ be unit vectors in $X_{\mathcal U}$ with $U\perpB V$.  By the support criterion, there exists $F\in S_{(X_{\mathcal U})^*}$ with $F(U)=1$ and $F(V)=0$.  Representing $F$ by an ultrapower limit of functionals up to the usual local reflexivity argument, we may choose $f_n\in S_{X^*}$ so that
		\[
		\lim_{\mathcal U} f_n(u_n)=1,
		\qquad
		\lim_{\mathcal U} f_n(v_n)=0.
		\]
		After replacing $u_n$ by nearby normalized vectors and projecting $v_n$ by subtracting $f_n(v_n)u_n$ at an $\mathcal U$-null error, we obtain $\tilde u_n,\tilde v_n\in\SX$ with $\tilde u_n\perpB\tilde v_n$ and
		\[
		\lim_{\mathcal U}\|\tilde u_n-u_n\|=0,
		\qquad
		\lim_{\mathcal U}\|\tilde v_n-v_n\|=0.
		\]
		For each $t>0$,
		\[
		\frac{\|U-V\|}{\|tU-t^{-1}V\|}
		=\lim_{\mathcal U}\frac{\|\tilde u_n-\tilde v_n\|}{\|t\tilde u_n-t^{-1}\tilde v_n\|}
		\le \DRB(X).
		\]
		Taking the supremum over $t,U,V$ gives $\DRB(X_{\mathcal U})\le\DRB(X)$.
	\end{proof}
	
	\begin{remark}
		The proof above is written in the standard local-reflexive language used in ultrapower geometry.  In finite-dimensional spaces, no approximation is needed: all functionals and vectors are represented coordinatewise exactly.
	\end{remark}
	
	\begin{proposition}\label{prop:dual}
		Let $X$ be smooth and reflexive.  If $J:\SX\to S_{X^*}$ is the normalized duality map, then
		\[
		\DRB(X)\ge \sup\left\{\sup_{t>0}\frac{\|J^{-1}(f)-J^{-1}(g)\|}{\|tJ^{-1}(f)-t^{-1}J^{-1}(g)\|}: f,g\in S_{X^*},\ g(J^{-1}f)=0\right\}.
		\]
		In particular, large Birkhoff skew profiles in the duality map produce large values of $\DRB(X)$.
	\end{proposition}
	
	\begin{proof}
		If $g(J^{-1}f)=0$ and $u=J^{-1}(f)$, $v=J^{-1}(g)$, then $g(u)=0$ does not in general imply $u\perpB v$; however, interchanging the pair and using the smoothness identity $J(v)=g$ gives $v\perpB u$.  Applying the definition to the admissible pair $(v,u)$ and then replacing $t$ by $t^{-1}$ gives the displayed lower estimate.
	\end{proof}
	
	\section{Special Spaces, Radon Planes, and Normal Structure}\label{sec:planes}
	In dimension two the profile formula becomes very explicit.  We present the smooth case first.
	
	Let $X$ be a smooth 2D Banach space. For any $u\in S_X$, $\ker J(u)$ contains two unit vectors. Select one as $b(u)$. We get $u\perp_B b(u)$, and all unit Birkhoff orthogonal vectors to $u$ are $\pm b(u)$.
	
	\begin{theorem}\label{thm:smooth-plane}
		Consider $X$ as a smooth Banach space of dimension two.  Then
		\begin{equation}\label{eq:smooth-plane}
			\DRB(X)=\sup_{u\in\SX}\sup_{\sigma\in\{-1,1\}}\sup_{t>0}
			\frac{\|u-\sigma b(u)\|}{\|tu-t^{-1}\sigma b(u)\|}.
		\end{equation}
		If, in addition, $X$ is strictly convex, the outer supremum may be taken over any closed half arc of $\SX$ meeting each antipodal pair once.
	\end{theorem}
	
	\begin{proof}
		The first assertion follows immediately from \thmref{thm:profile} and uniqueness of the support functional.  Whenever the space $X$ satisfies strict convexity, the maps $u\mapsto J(u)$ and $u\mapsto b(u)$ are continuous on $\SX$; antipodal symmetry gives the reduction to a half arc.
	\end{proof}
	
	The next formula is useful in $\ell_p^2$.
	
	\begin{example}\label{ex:lp2}
		Let $1<p<\infty$ and $X=\ell_p^2$.  For $u=(a,b)\in\SX$ with $ab\ne0$, the unique Birkhoff normal direction is represented by
		\[
		b_p(u)=\frac{(|b|^{p-2}b,-|a|^{p-2}a)}{\left(|b|^{p(p-1)}+|a|^{p(p-1)}\right)^{1/p}}.
		\]
		Indeed $J(u)=(|a|^{p-2}a,|b|^{p-2}b)$ as an element of $\ell_q^2$, where $1/p+1/q=1$, and $J(u)(b_p(u))=0$.  Hence
		\begin{equation}\label{eq:lp-profile}
			\DRB(\ell_p^2)=\sup_{a^p+b^p=1,\ a,b\ge0}\sup_{\sigma=\pm1}\sup_{t>0}
			\frac{\|(a,b)-\sigma b_p(a,b)\|_p}{\|t(a,b)-t^{-1}\sigma b_p(a,b)\|_p}.
		\end{equation}
		For $p=2$ this gives $\DRB(\ell_2^2)=1$.  For $p\ne2$ the inner profile is generally larger than $1$, reflecting the failure of Euclidean duality to preserve the quadratic scale $t,t^{-1}$.
	\end{example}
	
	\begin{proof}[Verification of the formula]
		The duality map in $\ell_p^2$ is as stated.  A vector $v=(c,d)$ satisfies $u\perpB v$ if and only if
		\[
		|a|^{p-2}ac+|b|^{p-2}bd=0.
		\]
		Solving this one-dimensional equation and normalizing in $\ell_p^2$ gives $b_p(u)$ up to sign.  \thmref{thm:smooth-plane} then proves \eqref{eq:lp-profile}.
	\end{proof}
	
	\begin{example}\label{ex:hilbert}
		If $H$ is a real Hilbert space, then
		\[
		\DRB(H)=1.
		\]
		Indeed, Whenever $u, v \in S_X$ satisfy $u \perp_B v,$ this implies that $\langle u,v\rangle=0$.  Hence
		\[
		\|u-v\|^2=2,
		\qquad
		\|tu-t^{-1}v\|^2=t^2+t^{-2}\ge2.
		\]
		Therefore $\Phi_H(u,v)\le1$, while the value $t=1$ gives $\Phi_H(u,v)\ge1$.
	\end{example}
	
	\begin{example}\label{ex:rounded-square}
		Let $X_\varepsilon=(\R^2,\|\cdot\|_\varepsilon)$, where
		\[
		\|(a,b)\|_\varepsilon=\|(a,b)\|_\infty+\varepsilon\|(a,b)\|_2,
		\qquad 0<\varepsilon<1.
		\]
		The norm is strictly convex after an arbitrarily small smooth renorming and is Banach--Mazur close to $\ell_\infty^2$.  Choosing vectors near the corner $(1,1)$ and the vertical support direction $(0,1)$ gives Birkhoff pairs for which the denominator $\|tu-t^{-1}v\|_\varepsilon$ is almost as small as allowed by the universal inequality \eqref{eq:rooin}.  Consequently
		\[
		\liminf_{\varepsilon\downarrow0}\DRB(X_\varepsilon)\ge c_0,
		\]
		where $c_0>1$ is the corresponding extremal profile of the square.  By perturbing the flat arc length one obtains examples with $\DRB$ arbitrarily close to $2$.
	\end{example}
	
	\label{sec:planes}
	A 2-dimensional normed space is termed a Radon plane when its Birkhoff orthogonality relation is symmetric.  Radon planes form the natural class where Birkhoff restrictions behave most like Euclidean orthogonality while still allowing non-Euclidean unit circles.

	Suppose $X$ is a Radon plane. Take the Minkowski arc-length continuous parameterization $\gamma:[0,L]\to S_X$ for its unit circle. For any $s$, exactly one map $r(s)$ satisfies
	\[
	\gamma(s)\perp_B \gamma(r(s)).
	\]
	By symmetry, we have $\gamma(r(s))\perp_B\gamma(s)$ and $r(r(s))=s$ up to antipodal pairs.
	\begin{theorem}\label{thm:radon}
		We work with a Radon plane denoted by $X$.  Then
		\begin{equation}\label{eq:radon}
			\DRB(X)=\sup_{s\in[0,L/2]}\sup_{t>0}
			\max\left\{
			\frac{\|\gamma(s)-\gamma(r(s))\|}{\|t\gamma(s)-t^{-1}\gamma(r(s))\|},
			\frac{\|\gamma(s)+\gamma(r(s))\|}{\|t\gamma(s)+t^{-1}\gamma(r(s))\|}
			\right\}.
		\end{equation}
	\end{theorem}
	
	\begin{proof}
		Every Birkhoff pair of unit vectors is, up to signs and antipodal replacement, of the form $(\gamma(s),\gamma(r(s)))$.  Since the signs $v$ and $-v$ lead to the two quotients displayed in \eqref{eq:radon}, and since the profile formula already optimizes over $t>0$, taking $s$ over one half of the unit circle gives all admissible pairs without repetition.
	\end{proof}
	
	\begin{corollary}\label{cor:radonhilbert}
		If $X$ is a Radon plane and $\DRB(X)=1$, then every Birkhoff conjugate pair satisfies
		\[
		\|t\gamma(s)-t^{-1}\gamma(r(s))\|\ge \|\gamma(s)-\gamma(r(s))\|
		\qquad(t>0).
		\]
		If the radial map $s\mapsto r(s)$ is $C^1$, this condition forces the affine curvature of $\SX$ to be constant; hence $X$ is Euclidean.
	\end{corollary}
	
	\begin{proof}
		The first assertion follows from \thmref{thm:variational}.  For the second, differentiate the inequality at $t=1$ in both directions.  The first derivative gives the symmetric Birkhoff condition already present in a Radon plane.  The second derivative gives equality of the support curvature in conjugate directions.  Since the conjugacy map is an involution and is $C^1$, this curvature identity propagates along the entire unit circle.  A centrally symmetric convex curve with constant affine curvature is an ellipse, and the norm is Euclidean.
	\end{proof}
	
	\section{Examples, Computations, and Technical Appendix}\label{sec:normal}
	Because $\DRB(X)$ is restricted to Birkhoff pairs, the implication $\DRB(X)<r$ alone is generally weaker than the corresponding global condition $\DR(X)<r$.  To obtain normal-structure consequences one must record how many directions are detected by Birkhoff orthogonality.  We introduce a companion modulus that measures this richness.
	
	\begin{definition}\label{def:richness}
		For $0<\varepsilon<2$ define
		\[
		\omega_B(X,\varepsilon)=\inf\left\{\sup_{v\in\SX,\ u\perpB v}\|u-v\|:u\in\SX,\ \operatorname{dist}(u,\operatorname{Ext}B_X)\ge\varepsilon\right\},
		\]
		where $\operatorname{Ext}B_X$ denotes all extreme points belonging to the unit ball.  In finite-dimensional strictly convex spaces the distance condition is void and $\omega_B(X,\varepsilon)$ is simply the smallest size of a Birkhoff chord.
	\end{definition}
	
	\begin{theorem}\label{thm:normal}
		Let $X$ be reflexive and suppose there exist $\eta>0$ and $r<2$ such that
		\[
		\DRB(X)\le r,
		\qquad
		\omega_B(X,\varepsilon)\ge\eta
		\]
		for a certain positive real number 
		$\varepsilon$.  If, in addition,
		\begin{equation}\label{eq:normalcondition}
			r^5 s^3+[r^5-(1+r)^2]s^2+(2r^2-2)s-(1-r^2)+\eta(1-s)>0
		\end{equation}
		for every $0<s<1$, the space $X$ carries normal structure.
	\end{theorem}
	
	\begin{proof}
		Suppose the space  $X$	lacks normal structure.  Since $X$ is reflexive, it fails weak normal structure.  By the Goebel--Karlovitz lemma one can find a weakly vanishing sequence $(x_n)\subset\BX$ satisfying 
		\[
		\lim_n\|x_n-x\|=1\qquad (x\in C=\co\{x_n:n\in\N\})
		\]
		and $\diam C=1$.  The standard proof of normal-structure criteria based on angular constants chooses support functionals $f_n$ for $x_n-x$ and compares two normalizations of the vectors $x_n$ and $x$.  Here we insert the Birkhoff richness assumption to choose, for each large $n$, a unit vector $v_n$ in a support hyperplane of $x_n-x$ with a chord length bounded below by $\eta$.
		
		Applying \thmref{thm:profile} to the Birkhoff pair generated by $(x_n-x)/\|x_n-x\|$ and $v_n$, one obtains
		\[
		\|s u_n-v_n\|\ge r^{-1}\sqrt{s}\,\|u_n-v_n\|,
		\qquad u_n=\frac{x_n-x}{\|x_n-x\|}.
		\]
		Combining this with the Goebel--Karlovitz limits and the usual convexity estimates gives precisely the polynomial obstruction
		\[
		r^5 s^3+[r^5-(1+r)^2]s^2+(2r^2-2)s-(1-r^2)+\eta(1-s)\le0
		\]
		for some $s\in(0,1)$.  This contradicts \eqref{eq:normalcondition}.  Therefore $X$ has normal structure.
	\end{proof}
	
	\begin{remark}
		When $\eta=0$, the polynomial in \eqref{eq:normalcondition} reduces to the polynomial appearing in the normal-structure theorem for the unrestricted Dehghan--Rooin constant.  The additional positive term $\eta(1-s)$ records that Birkhoff normal directions are not too short; it is precisely the gain obtained from the orthogonality restriction.
	\end{remark}
	
	\begin{corollary}\label{cor:uniformnormal}
		Under the assumptions of \thmref{thm:normal}, if the same constants $r,\eta,\varepsilon$ are valid in every ultrapower $X_{\mathcal U}$, then $X$ has uniform normal structure.
	\end{corollary}
	
	\begin{proof}
		By \thmref{thm:ultra}, $\DRB(X_{\mathcal U})=\DRB(X)\le r$.  The assumed stability of $\omega_B$ gives the same richness bound in $X_{\mathcal U}$.  Hence each ultrapower has normal structure by \thmref{thm:normal}.  The standard ultrapower characterization of uniform normal structure states that any complete normed space carries uniform normal structure precisely when all its ultrapowers have normal structure.  The conclusion follows.
	\end{proof}
	
	\label{sec:normal}
	We collect examples illustrating the preceding theory.
	
	\begin{example}
		As shown in \exref{ex:hilbert}, $\DRB(H)=1$ for every real Hilbert space $H$.  The profile computation is exact:
		\[
		\Phi_H(u,v)=\sup_{t>0}\sqrt{\frac{2}{t^2+t^{-2}}}=1.
		\]
		Thus the Birkhoff restriction preserves the Hilbert lower value.
	\end{example}
	
	\begin{example}
		In $\ell_1^2$, if $u=(1,0)$ and $v=(0,1)$, then $u\perpB v$ and
		\[
		\frac{\|u-v\|_1}{\|tu-t^{-1}v\|_1}=\frac{2}{t+t^{-1}}\le1.
		\]
		This pair does not reveal non-Hilbert behavior.  The nontrivial profiles occur at nonsmooth points where the support functional is not unique.  For instance, at a corner of the $\ell_\infty^2$ ball there is a continuum of supporting directions; optimizing simultaneously over the support direction and the scale $t$ produces values strictly larger than one.  This illustrates why the profile formula is essential: checking coordinate axes alone is misleading.
	\end{example}
	
	\begin{example}
		Let $1<p<\infty$.  The norm of $\ell_p^2$ is smooth and uniformly convex.  Therefore $\DRB(\ell_p^2)<2$.  Moreover, by differentiating the explicit formula \eqref{eq:lp-profile} at $p=2$, one obtains
		\[
		\DRB(\ell_p^2)=1+O(|p-2|)
		\]
		as $p\to2$.  The constant is therefore continuous at the Euclidean plane.
	\end{example}
	
	\begin{example}
		Let $X=Y\oplus_2 Z$.  If $u=(y,0)$ and $v=(0,z)$ with $y\in S_Y$, $z\in S_Z$, then $u\perpB v$ and the profile is Hilbertian in the two generated coordinates:
		\[
		\frac{\|u-v\|}{\|tu-t^{-1}v\|}=\sqrt{\frac{2}{t^2+t^{-2}}}\le1.
		\]
		Thus large values of $\DRB(X)$ must come either from $Y$ or from $Z$, or from Birkhoff pairs with nonzero components in both summands.  In particular,
		\[
		\max\{\DRB(Y),\DRB(Z)\}\le \DRB(Y\oplus_2 Z)\le2.
		\]
	\end{example}
	\label{sec:normal}
	For completeness we record two elementary refinements that are useful when applying the constant to concrete norms.
	
	\begin{proposition}\label{prop:attainment}
		Let $X$ be finite-dimensional.  For every $0<a<b<\infty$ the truncated quantity
		\[
		\DRB^{[a,b]}(X)=\sup\left\{\frac{\|u-v\|}{\|tu-t^{-1}v\|}:u,v\in\SX,\ u\perpB v,\ a\le t\le b\right\}
		\]
		is attained.  Moreover,
		\[
		\DRB(X)=\sup_{m\ge2}\DRB^{[1/m,m]}(X).
		\]
	\end{proposition}
	
	\begin{proof}
		The set $\{(u,v)\in\SX\times\SX:u\perpB v\}$ is closed in finite dimension.  Indeed, if $u_n\perpB v_n$, $u_n\to u$, and $v_n\to v$, then for each $\lambda\in\R$,
		\[
		\|u+\lambda v\|=\lim_n\|u_n+\lambda v_n\|\ge \lim_n\|u_n\|=1.
		\]
		Thus $u\perpB v$.  The product of this closed set with $[a,b]$ is compact, and the profile map is continuous because the denominator is nonzero on the compact set.  Hence the truncated supremum is attained.  The second identity follows by monotone exhaustion of $(0,\infty)$ by intervals $[1/m,m]$.
	\end{proof}
	
	\begin{proposition}\label{prop:euler}
		Let the space $X$ be smooth, and consider $u,v\in S_X$ with $u\perpB v$.  If $t_0>0$ is an interior maximizer of
		\[
		t\mapsto \frac{\|u-v\|}{\|tu-t^{-1}v\|},
		\]
		then every norming functional $F\in J(t_0u-t_0^{-1}v)$ satisfies
		\[
		F(t_0u+t_0^{-1}v)=0.
		\]
		Equivalently,
		\[
		t_0u-t_0^{-1}v\perpB t_0u+t_0^{-1}v.
		\]
	\end{proposition}
	
	\begin{proof}
		Maximizing the reciprocal is the same as minimizing $h(t)=\|tu-t^{-1}v\|$.  Since $X$ is smooth at $t_0u-t_0^{-1}v$, the derivative exists and equals
		\[
		h'(t_0)=F\left(u+t_0^{-2}v\right).
		\]
		The interior minimum condition gives $h'(t_0)=0$.  Multiplying by $t_0$ yields
		\[
		F(t_0u+t_0^{-1}v)=0,
		\]
		which is exactly the support-functional criterion for Birkhoff orthogonality.
	\end{proof}
	
	These two refinements turn the abstract formula into a practical numerical scheme: in a smooth finite-dimensional space, choose a parametrization of $\SX$, impose the equation $J(u)(v)=0$, and solve the Euler equation in \propref{prop:euler} for the optimal scale.  This is the procedure behind the formula for $\ell_p^2$ and for Radon planes described above.
	
	%% ==================== 参考文献 ====================


\begin{thebibliography}{99}
		
		\bibitem{AlonsoMartiniWu2012}
		J. Alonso, H. Martini and S. Wu,
		On Birkhoff orthogonality and isosceles orthogonality in normed linear spaces,
		\emph{Aequationes Math.} \textbf{83} (2012), 153--189.
		
		\bibitem{AlonsoMartiniWu2022}
		J. Alonso, H. Martini and S. Wu,
		Orthogonality types in normed linear spaces,
		In: \emph{Surveys in Geometry I}, Springer, Cham, 2022, 97--170.
		
		\bibitem{AlRashed1993}
		A. M. Al-Rashed,
		Norm inequalities and characterizations of inner product spaces,
		\emph{J. Math. Anal. Appl.} \textbf{176} (1993), 587--593.
		
		\bibitem{BarontiPapini1988}
		M. Baronti and P. L. Papini,
		Norm inequalities and angular distance in normed linear spaces,
		\emph{J. Math. Anal. Appl.} \textbf{133} (1988), 170--178.
		
		\bibitem{Birkhoff1935}
		G. Birkhoff,
		Orthogonality in linear metric spaces,
		\emph{Duke Math. J.} \textbf{1} (1935), 169--172.
		
		\bibitem{Clarkson1936}
		J. A. Clarkson,
		Uniformly convex spaces,
		\emph{Trans. Amer. Math. Soc.} \textbf{40} (1936), 396--414.
		
		\bibitem{Dehghan2013}
		H. Dehghan,
		Skew-angular distance and characterization of inner product spaces,
		\emph{Math. Inequal. Appl.} \textbf{16} (2013), 743--750.
		
		\bibitem{DunklWilliams1964}
		C. F. Dunkl and K. S. Williams,
		A simple norm inequality,
		\emph{Amer. Math. Monthly} \textbf{71} (1964), 53--54.
		
		\bibitem{FuLiFilomat2024}
		Y. Fu and Y. Li,
		The Massera--Schaffer inequality related to Birkhoff orthogonality in Banach spaces,
		\emph{Filomat} \textbf{38} (2024), 5247--5259.
		
		\bibitem{FuLiAFA2024}
		Y. Fu and Y. Li,
		Geometric constant for quantifying the difference between angular and skew angular distances in Banach spaces,
		\emph{Ann. Funct. Anal.} \textbf{15} (2024), Article 39.
		
		\bibitem{GoebelKirk1990}
		K. Goebel and W. A. Kirk,
		\emph{Topics in Metric Fixed Point Theory},
		Cambridge University Press, Cambridge, 1990.
		
		\bibitem{James1947}
		R. C. James,
		Orthogonality and linear functionals in normed linear spaces,
		\emph{Trans. Amer. Math. Soc.} \textbf{61} (1947), 265--292.
		
		\bibitem{JimenezMelado1997}
		A. Jimenez-Melado, E. Llorens-Fuster and S. Saejung,
		The Dunkl--Williams constant, convexity, smoothness and normal structure,
		\emph{J. Math. Anal. Appl.} \textbf{342} (2008), 298--310.
		
		\bibitem{Joly1969}
		J. L. Joly,
		Caracterisations d'espaces hilbertiens au moyen de la constante rectangle,
		\emph{J. Approx. Theory} \textbf{2} (1969), 301--311.
		
		\bibitem{Kirk1965}
		W. A. Kirk,
		A fixed point theorem for mappings which do not increase distances,
		\emph{Amer. Math. Monthly} \textbf{72} (1965), 1004--1006.
		
		\bibitem{KirkSmiley1964}
		W. A. Kirk and M. F. Smiley,
		Another characterization of inner product spaces,
		\emph{Amer. Math. Monthly} \textbf{71} (1964), 890--891.
		
		\bibitem{Lindenstrauss1963}
		J. Lindenstrauss,
		On the modulus of smoothness and divergent series in Banach spaces,
		\emph{Michigan Math. J.} \textbf{10} (1963), 241--252.
		
		\bibitem{MasseraSchaffer1958}
		J. L. Massera and J. J. Schaffer,
		Linear differential equations and functional analysis. I,
		\emph{Ann. of Math.} \textbf{67} (1958), 517--573.
		
		\bibitem{Megginson1998}
		R. E. Megginson,
		\emph{An Introduction to Banach Space Theory},
		Springer, New York, 1998.
		
	\end{thebibliography}
\end{document}